\documentclass[11pt,a4paper,reqno]{amsart}

\usepackage{setspace}
\usepackage{fullpage}
\usepackage{datetime}
\usepackage{xcolor}

\usepackage{enumitem}
\usepackage{amsmath}%
\usepackage{amsfonts}%
\usepackage{amssymb}%
\usepackage{graphicx}
\usepackage{mathrsfs}
\usepackage{hyperref}
\usepackage[margin=2.5cm]{geometry}

\usepackage[normalem]{ulem}

\usepackage{cite}

\usepackage[abbrev,msc-links,backrefs]{amsrefs}
\newtheorem{theorem}{Theorem}
\theoremstyle{plain}

\newtheorem{conjecture}[theorem]{Conjecture}

\newtheorem{fact}[theorem]{Fact}
\newtheorem{lemma}[theorem]{Lemma}

\newtheorem{proposition}[theorem]{Proposition}

\newtheorem{remark}[theorem]{Remark}
\numberwithin{theorem}{section}

\def\COMMENT#1{}
\let\COMMENT=\footnote
\usepackage{lineno} 

\begin{document}

\title{The threshold for powers of tight Hamilton cycles in random hypergraphs}

\author[Y.~Chang]{Yulin Chang}
\address{(Y. Chang) Data Science Institute, Shandong University, Jinan, 250100, China}
\email{ylchang@sdu.edu.cn}

\author[J.~Han]{Jie Han}
\address{(J. Han) School of Mathematics and Statistics and Center for Applied Mathematics, Beijing Institute of Technology, Beijing, 100000, China}
\email{han.jie@bit.edu.cn}

\author[L.~Sun]{Lin Sun}
\address{(L. Sun) School of Mathematics, Shandong University, Jinan, 250100, China}
\email{linsun77@163.com}

\thanks{Y.~Chang is partially supported by the China Postdoctoral Science Foundation (2022M711926, 2022T150386), Natural Science Foundation of China (12201352) and Natural Science Foundation of Shandong Province (ZR2022QA083). J.~Han is partially supported by Natural Science Foundation of China (12371341).}

\begin{abstract}
We investigate the occurrence of powers of tight Hamilton cycles in random hypergraphs.
For every $r\ge 3$ and $k\ge 1$, we show that there exists a constant $C > 0$ such that if $p=p(n) \ge Cn^{-1/\binom{k+r-2}{r-1}}$ then asymptotically almost surely the random hypergraph $H^{(r)}(n,p)$ contains the $k$th power of a tight Hamilton cycle.
This improves on a result of Parczyk and Person, who proved the same result under the assumption $p=\omega\left(n^{-1/\binom{k+r-2}{r-1}}\right)$ using a second moment argument.
\end{abstract}

\date{\today}

\maketitle


\section{Introduction}
For $r\geq 2$, an \emph{$r$-uniform hypergraph} (or \emph{$r$-graph}, for short) $H=(V,E)$ consists of a vertex set $V$ of order $n$ and an edge set $E$, where $E$ is a family of $r$-subsets of $V$, i.e., $E\subseteq \binom{V}{r}$.
When $E= \binom{V}{r}$, the $r$-graph $H$ is referred to as a \emph{complete $r$-graph}, denoted by $K_n^{(r)}$.
In the specific case of $r=2$, we simply refer to it as a \emph{graph} and denote it by $G=(V,E)$.

The investigation of Hamiltonicity and its associated problems is undoubtedly one of the most central and fruitful area in graph theory.
Notably, one of them is the classical Dirac's theorem~\cite{Dirac}, which asserts that every graph on $n$ $(n\geq 3)$ vertices with minimum degree at least $n/2$ contains a Hamilton cycle.
After a span of 20 years, Karp~\cite{K1972} demonstrated that it is NP-complete to determine whether a graph has a Hamilton cycle.
Nevertheless, there are numerous significant results which derive sufficient conditions for this property.
In recent years, there has been considerable attention directed towards generalizations of Dirac’s theorem and related problems in hypergraphs.
We recommend the surveys~\cites{RR2010,Zhao2016} to the reader for a comprehensive discussion on this subject.

The object of this research is to study Hamiltonicity in random hypergraphs.
In the realm of random graph theory, a fundamental question is to determine for which values of $p$ does the random hypergraph $H^{(r)}(n,p)$ asymptotically almost surely\footnote{We say that an event occurs \emph{asymptotically almost surely}, or \emph{a.a.s.}~for brevity, if the probability that it happens tends to $1$ as $n$ tends to infinity.} contain a copy of a given hypergraph $H$?
Here we extend our investigation to contain powers of tight Hamilton cylces.
Before stating it formally, it is necessary to provide the following definitions.

Given $r\geq 2$, we say that an $r$-graph $H$ contains a \emph{tight Hamilton cycle} if its vertices can be ordered cyclically such that every edge consists of $r$ consecutive vertices and every pair of consecutive edges (in the natural ordering of the edges) intersects in exactly $r-1$ vertices.
For $k\geq 1$, the \emph{$k$th power of a tight Hamilton cycle} (or \emph{$(r,k)$-cycle}, for short) denoted by $HC_{r}^k$ is an $r$-graph whose vertices can be ordered cyclically so that each consecutive $r+k-1$ vertices span a copy of $K^{(r)}_{r+k-1}$ and there are no other edges than the ones forced by this condition.
This extends the notion of tight Hamilton cycles in hypergraphs, which corresponds to the case $k=1$.
A \emph{random $r$-uniform hypergraph} $H^{(r)}(n,p)$ with vertex set $[n]:=\{1,\ldots,n\}$ is obtained by adding each possible $r$-subset of $[n]$ independently with probability $p=p(n)$. 
In the special case $r= 2$, we denote it by $G(n,p)$ which is the usual binomial \emph{random graph}.

Let us begin by discussing the case of binomial random graphs when $r=2$. 
For $k=1$, P\'{o}sa~\cite{Posa1} and Kor\v{s}hunov~\cite{Kors} independently showed that the threshold for Hamiltonicity in $G(n, p)$ is $\log n/n$.
This result was later improved by Koml\'{o}s and Szemer\'{e}di~\cite{KS1983}, who derived an exact formula for the probability of the existence of a Hamilton cycle.
For $k\ge 3$, it follows from a more general result of Riordan~\cite{R2000} that the threshold for the appearance of the $k$th power of a Hamilton cycle in $G(n,p)$ is $n^{-1/k}$, which was observed by K\"{u}hn and Othus~\cite{KO2012}.
However, the case $k=2$ is more challenging.
K\"{u}hn and Othus~\cite{KO2012} conjectured that the threshold for the square of a Hamilton cycle in $G(n, p)$ is $n^{-1/2}$, and they proved that $p\geq n^{-1/2+o(1)}$ is sufficient for the existence of the square of a Hamilton cycle in $G(n, p)$.
This bound was further improved by Nenadov and \v{S}kori\'{c}~\cite{NS2019}; Fischer, \v{S}kori\'{c}, Steger and Truji\'{c}~\cite{FSST22}; and finally, the threshold was recently determined by Kahn, Narayanan, and Park~\cite{KNP2021}, confirming the conjecture of K\"{u}hn and Othus~\cite{KO2012}.

Moving on to the case when $r\geq 3$, the study of Hamilton cycles in random hypergraphs is a more recent development.
Frieze~\cite{Frieze10} initiated the investigation by considering loose Hamilton cycles in random 3-uniform hypergraphs (the definition is omitted here, refer to~\cite{Frieze10} for more details).
Dudek and Frieze~\cite{DF2013} used a second moment argument to determine the sharp threshold for a tight Hamilton cycle in $H^{(r)}(n,p)$ is $e/n$ for $r \ge 4$.
In the case of $r = 3$, they showed that $H^{(3)}(n,p)$ contains a tight Hamilton cycle when $p=\omega(1/n)$. 
For $r\geq 3$ and $k\geq 2$, Parczyk and Person~\cite[Theorem 3.7]{PP2016} indicated that the threshold for the existence of the $k$th power of a tight Hamilton cycle in $H^{(r)}(n, p)$ is $n^{-\binom{r+k-2}{r-1}^{-1}}$.
More recently, the (sharp) threshold for nonlinear cycles are determined by Narayanan and Schacht~\cite{NS2020}.

In this note, we apply the idea of Narayanan and Schacht~\cite{NS2020} to further compute the threshold for the appearance of the $k$th power of a tight Hamilton cycle in $H^{(r)}(n, p)$, which improves upon the result of Parczyk and Person~\cite{PP2016}.
As in~\cite{NS2020}, this requires a second moment estimate together with a powerful theorem of Friedgut~\cite{Friedgut2005} characterizing coarse thresholds.

\begin{theorem}\label{main}
    Let $r\ge 3$ and $k\ge 1$ be integers. 
    There exists a constant $C>0$ such that if $p\ge Cn^{-1/\binom{k+r-2}{r-1}}$ then $H^{(r)}(n,p)$ a.a.s.~contains the $k$th power of a tight Hamilton cycle.
\end{theorem}

The initial motivation of this project is to make progress on a conjecture in~\cite{KNP2021}, who conjectured the following.
\begin{conjecture}[{\cite[Conjecture 1.3]{KNP2021}}]
For fixed $\varepsilon > 0$ and $p\ge (1+\varepsilon)\sqrt{e/n}$, $G(n, p)$ a.a.s.~contains a square of a Hamilton cycle.
\end{conjecture}

The (conjectued) threshold $\sqrt{e/n}$ is obtained by standard first-moment method.
The main result of~\cite{KNP2021} gives a universal constant $K$ and our goal was to improve it to a computable one.
However, our proof of Theorem~\ref{main} fails for $r=2$ (see the proof of Lemma~\ref{k}).
Nevertheless, it seems very likely that the proof method in~\cite{KNP2021} applies in our case but we prefer the current one that gives an explicit constant $C= 4er^{2}\binom{k+r-2}{r-1}((2k+2r-3)e)^{\binom{k+r-2}{r-1}^{-1}}$.

The organisation of this paper is as follows. 
In Section 2 we give some fundamental tools and properties of power of Hamilton cycles, which will be used to prove our key lemmas and Theorem~\ref{main}. 
In Section 3 we provide the main lemma (Lemma~\ref{lem:expect}) and the proof of Theorem~\ref{main}. 
In Section 4 we prove Lemma~\ref{lem:expect}.

\section{Preliminaries}
The following probability inequalities will be useful in our proof, which can be found, e.g., in [12, Corollary 2.3].
\begin{lemma}[Paley–Zygmund inequality]\label{lem:PZ}
    If $X$ is a nonnegative random variable, then
\[\mathbb{P}(X>0)\ge \frac{\mathbb{E}[X]^2}{\mathbb{E}[X^2]}.\]
\end{lemma}

We need the celebrated machinery of Friedgut~\cite{Friedgut2005} to say that the property of containing the $k$th power of a tight Hamilton cycle has a sharp threshold.
\begin{lemma}[{\cite{Friedgut2005}}]\label{lem:coarse}
Fix $r\in \mathbb{N}$ and let $W =(W_n)_{n\ge 0}$ be a monotone $r$-graph property that has a coarse threshold. 
Then there exists a constant $\alpha>0$, a threshold function $\hat{p} =\hat{p}(n)$ with
\[\alpha< \mathbb{P}(H^{(r)}(n, \hat{p})\in W_n) < 1 -\alpha\]
for all $n\in \mathbb{N}$, a constant $\beta>0$ and a fixed $r$-graph $F$ such that the following holds: for infinitely many $n\in \mathbb{N}$, there exists an $r$-graph on $n$ vertices $H_n\notin W_n$ such that
\[\mathbb{P}(H_n\cup H^{(r)}(n, \beta\hat{p})\in W_n) < 1-2\alpha,\]
where the random $r$-graph $H^{(r)}(n, \beta\hat{p})$ is taken to be on the same vertex set as $H_n$, and
\[\mathbb{P}(H_n\cup \tilde F \in W_n) > 1-\alpha,\]
where $\tilde F$ denotes a random copy of $F$ on the same vertex set as $H_n$.
\end{lemma}

We shall also require the following three results about the properties of $HC^{k}_{r}$.
Given a $r$-graph $H=(V,E)$ and a $d$-subset $S\subseteq V$ with $1\leq d\leq k-1$. 
We define the \emph{degree} of $S$, denoted by $\deg_H(S)$, to be the number of edges in $H$ containing $S$, that is, $\deg_H(S)=\left|\left\{e\in E\colon\, e\supseteq S\right\}\right|$.
The \emph{maximum $d$-degree $\Delta_{d}(H)$} of $H$ is the maximum of $\deg_H(S)$ over all $d$-subsets $S$ of $V$.
We refer to $\Delta(H):=\Delta_{1}(H)$ as the \emph{maximum vertex degree} of $H$.
\begin{fact}\label{max}
Let $r\ge 2$ and $k\ge 1$ be integers. Then $\Delta(HC^{k}_{r})=r\binom{k+r-2}{r-1}$.
\end{fact}
\begin{proof}
Suppose that $(v_1,\ldots,v_n)$ is a cyclic ordering of $HC^{k}_{r}$ such that each consecutive $r+k-1$ vertices span a copy of $K^{(r)}_{r+k-1}$.
For any $v_i\in V( HC^{k}_{r})$, we have 
\[\deg_{ HC^{k}_{r}}(v_i)=\binom{k+r-2}{r-1}+(k+r-2)\binom{k+r-3}{r-2}=r\binom{k+r-2}{r-1},\]
where we collect all edges containing $v_i$ as the last vertex following the ordering and all edges containing $v_i$ and $v_j$ such that $v_j$ is the last vertex following the ordering for each $j\in \{i+1,\ldots,i+r+k-2\}$, of course with indices being considered cyclically modulo $n$.
\end{proof}

\begin{fact}\label{f4}
Let $r\ge 2$ and $k\ge 1$ be integers. 
Suppose	that $P$ is a subhypergraph of $HC^{k}_{r}$ with $b$ edges and $s$ components.
Then $|V(P)|\ge b\binom{k+r-2}{r-1}^{-1}+(r-1) s$.
\end{fact}
\begin{proof}
Suppose that $P$ contains $s$ components $P_1,\ldots, P_s$ with $|E(P_i)|=b_i$ for $i\in[s]$.
Note that $\sum_{i\in [s]}b_{i}=b$ and $|V(P_i)|\ge r$ since $P_i$ is a component. 
We claim $|V(P_i)|\ge b_{i}\binom{k+r-2}{r-1}^{-1}+r-1$ for each $i\in[s]$.
If $|V(P_i)|\ge k+r-1$, then
\[
\begin{aligned}
	b_i &\leq\binom{k+r-1}{r}+\binom{k+r-2}{r-1}\left(|V(P_i)| -(k+r-1)\right).
\end{aligned}
\]
So $|V(P_i)|\ge b_{i}\binom{k+r-2}{r-1}^{-1}-\binom{k+r-1}{r}\binom{k+r-2}{r-1}^{-1}+ k+r-1\ge b_{i}\binom{k+r-2}{r-1}^{-1}+r-1$.
On the other hand, if $r \leq |V(P_i)| \leq k+r-2$, then
\[\ b_i \leq\binom{|V(P_i)|}{r}
\leq(|V(P_i)|-r+1)\binom{|V(P_i)|}{r-1}\le(|V(P_i)|-r+1)\binom{k+r-2}{r-1}.\] 
Thus in both cases, we have $|V(P_i)|\ge b_{i}\binom{k+r-2}{r-1}^{-1}+r-1$.
Summing over all $i\in[s]$, we have the lower bound
$|V(P)|\ge b\binom{k+r-2}{r-1}^{-1}+(r-1)s$.
\end{proof}

\begin{lemma}\label{l1}
Let $r\ge 2$ and $k\ge 1$ be integers. 
Suppose $H$ is a copy of the $k$th power of a tight Hamilton cycle of order $n$.
Then the number of connected subhypergraphs of $H$ containing a given vertex with $b$ edges is at most $e^{b}r^{2b}\binom{k+r-2}{r-1}^{b}$.
\end{lemma}
\begin{proof}
Fix $v\in V(H)$.
We greedily construct a rooted tree $T$ on vertex set $\{v\}\cup E(H)$ with root $v$ as follows, which is permitted to have multiple vertices. 
We start with the empty tree and add the root vertex $v$ to it.
Then we add all edges containing $v$ in $H$ to $T$ as the children of $v$.
Suppose we have $e_i\in V(T)$.
If there is $e_j\in E(H)$ such that $e_i\cap e_j\neq\emptyset$, then we add $e_j$ as a child of $e_i$. 
Furthermore, an edge $e\in E(H)$ is a leaf of $T$ if and only if $e$ appears twice in the only path from $v$ to this vertex.
So $T$ is a finite graph.
Note that $\Delta(T)\le r\Delta(H)=r^{2}\binom{k+r-2}{r-1}$ by Fact~\ref{max}.
Now we use the following fact to estimate the number of desired subhypergraphs.

\begin{fact}[{\cite{KNP2021}}]\label{ns}
For a graph $G$ with maximum degree $\Delta$, the number of connected subgraphs of $G$ containing a given vertex with $b$ edges is at most $(e\Delta)^b$.
\end{fact}

By the definition of $T$, for each connected subhypergraph $H'$ of $H$ containing $v$ with $b$ edges, there exists at least one subtree of $T$ starting from $v$ on $\{v\}\cup E(H')$ with $b$ edges.
Indeed, since $H'$ is connected, we can suppose $E(H')=\{e_{01},\dots,e_{0{i_0}},e_{11},\dots,e_{1{i_1}},e_{21},\dots,e_{2{i_2}},\dots,e_{j1},\dots,e_{j{i_j}}\}$ such that $v\in e_{0m}$ for $m\in[i_0]$, $i_0+i_1+\dots+i_j=b$, and for $s\ge 1$, each of $e_{s1},\dots,e_{s{i_s}}$ has nonempty intersection with some $e_{s-1p}$, $p\in[i_{s-1}]$ and 
is disjoint from all $e_{qp}$, $p\in[i_{q}], q<s-1$.
So there exists $T'\subseteq T$ on $\{v\}\cup E(H')$ rooted on $v$ such that each of $e_{s1},\dots,e_{s{i_s}}$ is a child of some $e_{s-1p}$, $p\in[i_{s-1}]$.
By Fact~\ref{ns}, the number of subtrees starting rooted at $v$ with $b$ edges is at most $e^br^{2b}\binom{k+r-2}{r-1}^{b}$.
So the number of connected subhypergraphs of $H$ containing $v$ with $b$ edges is at most $e^br^{2b}\binom{k+r-2}{r-1}^{b}$.
\end{proof}

Finally, we collect some standard estimates needed in our proof.
\begin{fact}\label{f1}
For all $n\in\mathbb{N}$, we have \[n !\ge \left(\frac{n}{e}\right)^n\] and for all positive integers $1 \leq k \leq n$, we have
\[\binom{n}{k}\leq\left(\frac{e n}{k}\right)^k.\]
\end{fact}

\begin{fact}\label{f3}
For integers $1\le x\le n$, we have $\frac{(n-x)!}{n!} \le \left(\frac{e}{n}\right)^x$.
\end{fact}
\begin{proof}
We first claim that $n(n-1)\cdots (n-x+1)\ge\left(\frac{n}{e}\right)^x$.
Indeed,
\[\prod_{a=0}^{x-1}(n-a)\ge\prod_{a=0}^{x-1}\frac{n}{x}(x-a)=\left(\frac{n}{x}\right)^xx!\ge\left(\frac{n}{x}\right)^x\left(\frac{x}{e}\right)^x=\left(\frac{n}{e}\right)^x\]
where we used $x\le n$ in the first inequality and Fact~\ref{f1} in the last inequality.
So we have 
\[
\begin{aligned}
\frac{(n-x)!}{n!}=\frac{1}{n(n-1)\cdots (n-x+1)}\le \left(\frac{e}{n}\right)^x.
\end{aligned}
\]
\end{proof}

\section{Proof of the main theorem}
\label{s4}
Given $\ell\ge r\ge 2$, an \emph{$r$-uniform tight path} on $\ell$ vertices is the $r$-graph whose vertices are $\{v_1,\ldots, v_{\ell}\}$ and its edges are $\{v_i,\ldots, v_{i+r-1}\}$ for all $i \in\{1,\ldots,\ell-r+1\}$.
For $k\geq 1$, we say that an $r$-graph is the \emph{$k$th power of an $r$-uniform tight path} (or an \emph{$(r,k)$-path}, for short) if its vertices can be ordered such that each consecutive $r+k-1$ vertices span a copy of $K^{(r)}_{r+k-1}$, and there are no other edges than the ones forced by this condition.
This extends the notion of (tight) paths in hypergraphs, which corresponds to the case $k=1$.

Another key component of our proof is the following (loose) second moment estimate.

\begin{lemma}\label{lem:expect}
Let $r\ge 3$ and $k\ge 1$ be integers.
Suppose $C\ge 4er^{2}\binom{k+r-2}{r-1}((2k+2r-3)e)^{\binom{k+r-2}{r-1}^{-1}}$ and $p= Cn^{-1/\binom{k+r-2}{r-1}}$. 
Then we have $\mathbb{E}[X^2] = O_{r,k}(\mathbb{E}[X]^2)$.
Moreover, if $C \to \infty$ when $n\to \infty$, then $\mathbb{E}[X^2]\le(1+o(1))\mathbb{E}[X]^2$.
\end{lemma}

Now we are ready to prove Theorem~\ref{main} by applying Lemmas~\ref{lem:PZ},~\ref{lem:coarse}, and~\ref{lem:expect}, which follows the proof idea used in~\cite{NS2020}.

\begin{proof}[Proof of Theorem~\ref{main}]
Let $r\ge 3$ and $k\ge 1$ be given integers, and let $C>0$ be the constant provided in Lemma~\ref{lem:expect}.
Let us denote $p^*_{k,r}=p^*_{k,r}(n):=n^{-1/{\binom{k+r-2}{r-1}}}$, and assume $p\ge Cp^*_{k,r}$.
Our goal is to show that a.a.s.~the random $r$-graph $H^{(r)}(n, p)$ contains the $k$th power of a tight Hamilton cycle.
To accomplish this, we divide the proof into two parts.
Firstly, we show that the property of containing the $k$th power of a tight Hamilton cycle has a sharp threshold.
Subsequently, we show that this sharp threshold must be $p^*_{k,r}$.

Recall that $p\ge Cp^*_{k,r}$, where $p^*_{k,r}=n^{-1/{\binom{k+r-2}{r-1}}}$ and $C>0$ is the constant mentioned in Lemma~\ref{lem:expect}. 
By applying Lemmas~\ref{lem:PZ} and~\ref{lem:expect}, it can be deduced that the random $r$-graph $H^{(r)}(n, p)$ contains the $k$th power of a tight Hamilton cycle with a probability of at least $\delta$, where $\delta=\delta(r,k)>0$.
This observation implies that if the property of containing the $k$th power of a tight Hamilton cycle has a sharp threshold, then the sharp threshold is necessarily asymptotic to $p^*_{k,r}$.
So it remains to prove that the monotone $r$-graph property $W =(W_n)_{n\ge 0}$ of containing the $k$th power of a tight Hamilton cycle has a sharp threshold.

Now, let us proceed by assuming the contrary, namely that $W$ has a coarse threshold.
According to Lemma~\ref{lem:coarse}, there exist universal constants $\alpha,\beta>0$, a threshold function $\hat{p} =\hat{p}(n)$, and a fixed $r$-graph $F$ such that for infinitely many $n\in \mathbb{N}$, there exists an $n$-vertex $r$-graph $H_n\notin W_n$ such that the addition of a random copy of $F$ to $H_n$ significantly increases the likelihood of the resulting graph having the property $W_n$ compared to the addition of a random collection of edges with a density of approximately $\hat{p}$; more precisely, we have
\[
\mathbb{P}(H_n\cup H^{(r)}(n, \beta\hat{p})\in W_n)<1-2\alpha, \label{thm:equ1}
\]
where $H^{(r)}(n, \beta\hat{p})$ denotes the random $r$-graph with the same vertex set as $H_n$, and
\[
\mathbb{P}(H_n\cup \tilde{F} \in W_n)>1-\alpha, \label{thm:equ2}
\]
where $\tilde{F}$ denotes a random copy of $F$ on the same vertex set as $H_n$.

Note that the only way $F$ can help induce the $k$th power of a tight Hamilton cycle in $H_n$ is through some subhypergraph of itself that appears in all large enough $k$th power of tight Hamilton cycles. 
By the pigeonhole principle (and adding extra edges if necessary), we conclude from~(\ref{thm:equ2}) that there exists a fixed $(r,k)$-path $P$, say with $e_P$ edges on $v_P$ vertices, such that, for some universal constant $\gamma>0$, we have
\[\mathbb{P}(H_n\cup \tilde{P} \in W_n) > \gamma,\]
where $\tilde{P}$ denotes a random copy of $P$ on the same vertex set as $H_n$. 
In other words, a positive fraction of all the possible ways to embed $P$ into the vertex set of $H_n$ are \emph{useful} and end up completing the $k$th power of a tight Hamilton cycle.
Moreover, we can assume that $v_P\ge k+r-1$ because we can add extra vertices if necessary, then we have $e_P=\binom{k+r-1}{r}+\left(v_P-(k+r-1)\right)\binom{k+r-2}{r-1}=\left(v_P-\frac{(r-1)(k+r-1)}{r}\right)\binom{k+r-2}{r-1}$ by the definition of the $(r,k)$-path.

From Lemma~\ref{lem:coarse} we know that $\hat{p}$ is an asymptotic threshold for $W$, clearly $\hat{p}=\Theta(p_{k,r}^{*})=\Theta\left(n^{-1/{\binom{k+r-2}{r-1}}}\right)$, since $p_{k,r}^{*}$ is also an asymptotic threshold for $W$, as can be read off from the proof of Lemma~\ref{lem:expect}. 
Next, we consider the number of useful copies of $P$, denoted by $Y$, created by adding a $\beta \hat{p}= \Theta\left(n^{-1/{\binom{k+r-2}{r-1}}}\right)$ density of random edges to $H_n$. Then
\[\mathbb{E}[Y]\ge \gamma\binom{n}{v_P}\frac{v_P!}{aut(P)}(\beta\hat{p})^{e_P}=\Omega\left(n^{v_P}\hat{p}^{e_P}\right),\]
where $aut(P)$ is the number of automorphisms of $P$.
Now we wish to apply the results from~\cite[Propositions 2.1 and 2.3]{BHKM2019} and Janson's inequality (see e.g.~\cite[Theorem 2.14]{Janson} and also~\cite[inequalty (2)]{BHKM2019}) to prove that $\mathbb{P}(Y=0)\rightarrow 0$ as $n\rightarrow \infty$.
For this denote by $Y'$ the number of copies of $P$ in $H^{(r)}(n, \beta\hat{p})$, then the result~\cite[Proposition 2.1]{BHKM2019} gives that\footnote{Here we omit the definitions of $\Delta_Y$, $\Delta_{Y'}$, and $\Phi_P$, and we refer the reader to \cite{BHKM2019} for more details.} $\Delta_Y\le \Delta_{Y'}\le v_P!2^{2v_P}(\beta \hat{p})^{2e_P}/\Phi_P$, where the first inequality follows straightly by the definitions of $\Delta_Y$ and $\Delta_{Y'}$. 
Then by taking $t=\mathbb{E}[Y]$ in Janson's inequality (\cite[inequalty (2)]{BHKM2019}), and together with the result \cite[Proposition 2.3]{BHKM2019}, yields
\[\mathbb{P}(Y=0)\le \exp\left(-\frac{\mathbb{E}[Y]^2}{2\Delta_Y}\right) \le\exp(-\Theta(n))\rightarrow 0\] 
when $n\rightarrow \infty$, as required.
This implies that adding a $\beta\tilde{p}$ density of random edges to $H_n$ must a.a.s.~create at least one useful copy of $P$ in $H_n$ and complete the $k$th power of a tight Hamilton cycle, contradicting~(\ref{thm:equ1}).
Hence, we conclude that $W$ has a sharp threshold, which completes the proof.
\end{proof}

\section{Proof of Lemma~\ref{lem:expect}}
Let $Q_{n}$ be the symmetric group of permutations of $[n]$, where
a permutation $\sigma\in Q_{n}$ is an arrangement $\sigma(1), \sigma(2),\dots, \sigma(n)$ of the elements of $[n]$.
Note that $|Q_{n}|=n!/(2n)=(n-1)!/2$.
Let $m$ be the number of edges in the $k$th prower of a Hamilton tight cycle of order $n$.
Then $m=\binom{k+r-2}{r-1}n$.
Given $\sigma\in Q_{n}$, consider the $r$-graph $H_{\sigma}$ on $[n]$ with $m$ edges, where for $i\in[n]$, the $(k+ r-1)$-set $\{\sigma(i), \sigma(i + 1), \dots, \sigma(i +k+ r-2)\}$ induces a clique $K^{(r)}_{k+ r-1}$, of course with indices being considered cyclically modulo $k+r-1$. 
We write $H_{\sigma}$ for such natural $k$th prower of a Hamilton tight cycle in $K^{(r)}_{n}$ associated with $\sigma$.

For $0 \le b \le m$ and $\sigma\in Q_{n}$, let $N_{\sigma}(b)$ denote the number of
permutations $\tau\in Q_{n}$ such that $|E(H_{\sigma}\cap H_{\tau})|=b$.
In order to prove Lemma~\ref{lem:expect}, we need the following key lemma.
\begin{lemma}\label{k}
Let $r\ge 3$ and $k\ge 1$ be integers.
Suppose $C\ge 4er^{2}\binom{k+r-2}{r-1}((2k+2r-3)e)^{\binom{k+r-2}{r-1}^{-1}}$ and $p=Cn^{-1/\binom{k+r-2}{r-1}}$.
For any $\sigma\in Q_{n}$, we have 
\[\frac{\sum_{b=1}^{m}N_{\sigma}(b)p^{-b}}{|Q_n|}=O_{r,k}(1).\]
Moreover, if $C \to \infty$ when $n\to \infty$, then we have $\frac{\sum_{b=1}^{m}N_{\sigma}(b)p^{-b}}{|Q_n|} = o(1)$.
\end{lemma}
Now we prove Lemma~\ref{lem:expect} by assuming that Lemma~\ref{k} holds. 
\begin{proof}[Proof of Lemma~\ref{lem:expect}]
Let $H:=H^{(r)}(n, p)$.
We define the random variable $X$ that counts the number of $\sigma\in Q_{n}$ for which
the $k$th power of Hamilton tight cycle $H_{\sigma}$ is contained in $H$.
It is easy to get that $\mathbb{E}[X]=|Q_n|p^m=(n-1)!p^m/2$.
Now we estimate the second moment of $X$.
We see that 
\[
\begin{aligned} 
\mathbb{E}[X^2]=\sum_{\sigma,\tau \in Q_n} \mathbb{P}(H_{\sigma}\cup H_{\tau} \subseteq H) 
&=\sum_{\sigma\in Q_n} \left(\mathbb{P}(H_{\sigma} \subseteq H)\sum_{\tau \in Q_n } \mathbb{P}(H_{\tau} \subseteq H\,|\,H_{\sigma} \subseteq H)\right)
\\&=\sum_{\sigma\in Q_n} \left(p^{m}\sum_{b=0}^{m}N_{\sigma}(b)p^{m-b}\right)
\\&=\sum_{\sigma\in Q_n}N_{\sigma}(0)p^{2m}+\sum_{\sigma\in Q_n} \left(p^{2m}\sum_{b=1}^{m}N_{\sigma}(b)p^{-b}\right)
\\&\le| Q_n|^2p^{2m}+\sum_{\sigma\in Q_n} \left(p^{2m}\sum_{b=1}^{m}N_{\sigma}(b)p^{-b}\right).
\end{aligned}
\]
Using Lemma~\ref{k}, we have
\[
\begin{aligned} 
\frac{\mathbb{E}[X^2]}{\mathbb{E}[X]^2}&\le\dfrac{| Q_n|^2p^{2m}+\sum_{\sigma\in Q_n} \left(p^{2m}\sum_{b=1}^{m}N_{\sigma}(b)p^{-b}\right)}{| Q_n|^2p^{2m}}
\\&=1+\frac{1}{|Q_n|}\times\sum_{\sigma\in Q_n}\frac{\sum_{b=1}^{m}N_{\sigma}(b)p^{-b}}{|Q_n|}
=O_{r,k}(1),
\end{aligned}
\]
completing the proof of Lemma~\ref{lem:expect}.
Moreover, if $C \to \infty$ when $n\to \infty$, then $\mathbb{E}[X^2]\le(1+o(1))\mathbb{E}[X]^2$.
\end{proof}

\subsection{Proof of Lemma~\ref{k}}
For $\sigma\in Q_{n}$ and $1\le s\le b\le m$, we denote $N_{\sigma}(b,s)$ by the number of $\tau\in Q_{n} $such that 
$H_{\sigma}\cap H_{\tau}$ has $b$ edges and $s$ components.
Note that $N_{\sigma}(b)=\sum_{s=1}^{b}N_{\sigma}(b,s)$.
In order to compute $N_{\sigma}(b,s)$, we need the following two propositions.
\begin{proposition}\label{p2}
For any $\sigma\in Q_{n}$ and $1\le s\le b$, fix a subhypergraph $P$ of $H_{\sigma}$ with $b$ edges and $s$ components.
Then the number of permutations $\tau\in Q_{n}$ with $P \subseteq H_{\tau}$ is at most $(n-|V(P)|+s-1)!(2k+2r-4 ))^{|V(P)|-s}/2$.
\end{proposition}
\begin{proof}
Suppose $P_1,\ldots,P_s$ are the components of $P$.
Fix a vertex $v_i\in P_i$ as the root vertex of $P_i$ for each $i\in[s]$.
We first specify a cyclic permutation of $\{v_1,\dots,v_{s}\}\cup(V(K_{n}^{r})\setminus V(P))$.
The number of ways to do this is at most $(n-|V(P)|+s)!/(2(n-|V(P)|+s))=(n-|V(P)|+s-1)!/2$.

Next we insert the vertices of $V(P)\setminus\{v_1,\dots,v_{s}\}$ to extend the cyclic permutation to a full cyclic ordering of $V(K_{n}^{(r)})$.
For $i\in[s]$, let $|V(P_i)|=p_i$, then $|V(P)|=\sum_{i=1}^{s}p_i$. 
Now we label vertices of $V(P_i)$ as $u_i^1,u_i^2,\ldots,u_i^{p_i}$ such that $u_i^1:=v_i$ and for $y\ge2$ there exist $x<y$ and $e\in E(P_{i})$ with $\{u_i^x,u_i^y\}\subseteq e$, that is, $u_i^y$ is a neighbor of $u_i^x$ in $P_i$.
Suppose we have inserted $u_i^1,u_i^2,\ldots,u_i^{y-1}$. 
Note that there exists $u_i^x$ for some $x\in[y-1]$ such that $u_i^y$ is a neighbor of $u_i^x$ in $P_{i}$. 
Then there are at most $2(k+r-2)$ places to insert $u_i^y$.
So the number of possibilities to insect $V(P_{i})\setminus\{v_i\}$ is at most $(2k+2r-4)^{p_i-1}$.
Summing over all $i\in[s]$, we get that the number of ways to extend the cyclic permutation to a full cyclic ordering of $V(K_{n}^{(r)})$ (which is our desired permutation of $Q_{n}$) is at most $\prod_{i=1}^{s}(2k+2r-4)^{p_i-1}=(2k+2r-4)^{|V(P)|-s}$.
Together with the number of ways of embedding a cyclic permutation of $\{v_1,\dots,v_{s}\}\cup(V(K_{n}^{r})\setminus V(P))$, we conclude that the number of $\tau\in Q_{n}$ with $P \subseteq H_{\tau}$ is at most $(n-|V(P)|+s-1)!(2k+2r-4)^{|V(P)|-s}/2$.
\end{proof}
\begin{proposition}\label{p1}
For any fixed $\sigma\in Q_{n}$ and $1\le s\le b$, the number of subhypergraphs of $H_{\sigma}$ with $b$ edges and $s$ components is at most $\binom{n}{s}\binom{b-1}{s-1}e^br^{2b}\binom{k+r-2}{r-1}^{b}$.
\end{proposition}

\begin{proof}
To specify a subhypergraph $P$ of $H_{\sigma}$ with $b$ edges and $s$ components as in the proof of Proposition~\ref{p2}, we proceed as follows. 
We first choose root vertices $v_{1},\ldots,v_s$ for the components $P_1,\ldots,P_s$ of $P$, respectively. 
Note that the number of choices for $v_{1},\ldots,v_s$ is at most $\binom{n}{s} $.
We then specify the size of $P_i$, denoted by $b_i$, for each $i\in[s]$. 
So the number of possibilities is equivalent to the number of positive integer solutions of $\sum_{i=1}^{s}b_i=b$, which is at most $\binom{b-1}{s-1} $.
Finally, we specify a connected subhypergraph $P_i$ containing $v_i$ with $b_i$ edges for each $i\in[s]$. 
By Lemma~\ref{l1} with $H_{\sigma}$ in place of $H$, there are at most $e^{b_{i}}r^{2b_{i}}\binom{k+r-2}{r-1}^{b_{i}}$ possibilities for this.
Combining these estimates, we get that the number of subhypergraphs of $H_{\sigma}$ with $b$ edges and $s$ components is at most \[\binom{n}{s}\binom{b-1}{s-1}\prod_{i=1}^{s}e^{b_{i}}r^{2b_{i}}\binom{k+r-2}{r-1}^{b_{i}}= \binom{n}{s}\binom{b-1}{s-1}e^br^{2b}\binom{k+r-2}{r-1}^{b}.\]
\end{proof}
Now we are ready to prove Lemma~\ref{k}.
\begin{proof}[Proof of Lemma~\ref{k}]
Fix $\sigma\in Q_{n}$.
For any $1\le s\le b$, by Proposition~\ref{p1}, the number of subhypergraphs of $H_{\sigma}$ with $b$ edges and $s$ components is at most $\binom{n}{s}\binom{b-1}{s-1}e^br^{2b}\binom{k+r-2}{r-1}^{b} \le (en/s)^s (2er^2)^b \binom{k+r-2}{r-1}^{b}$.
Considering each subhypergraph $P$ with $b$ edges and $s$ components, by Proposition~\ref{p2}, the number of $\tau\in Q_{n}$ with $P \subseteq H_{\tau}$ is at most $(n-|V(P)|+s-1)!(2k+2r-4))^{|V(P)|-s}/2$.
Let $x:=|V(P)|-s$ and note that $r\ge 3$, then we have $x\ge b\binom{k+r-2}{r-1}^{-1}+(r-2) s \ge b\binom{k+r-2}{r-1}^{-1}+ s$ by Fact~\ref{f4}.
So by the definition of $N_{\sigma}(b,s)$ and writing $c:=2er^{2}\binom{k+r-2}{r-1}$, we obtain
\[
\begin{aligned}\label{e1}
 N_{\sigma}(b,s)\le c^b (en/s)^s (n-x-1)!(2k+2r-4 )^{x}/2.
 \end{aligned}\tag{2}
 \]

By Fact~\ref{f3} and the lower bound on $x$, we have 
\[\label{e2}\frac{(n-x-1)!}{(n-1)!}(2k+2r-4 )^{x}\le\left(\frac{e}{n-1}\right)^x(2k+2r-4 )^{x}\le\left(\frac{(2k+2r-3 )e}{n}\right)^{b\binom{k+r-2}{r-1}^{-1}+s}, 
\]
where we used that $\frac{2k+2r-4}{n-1}\le \frac{2k+2r-3}{n}$. 
Since $p= Cn^{-1/\binom{k+r-2}{r-1}}$, we get
\[
\begin{aligned} 
2^b \frac{2N_{\sigma}(b,s)p^{-b}}{(n-1)!}&\le\frac{2^b c^b n^{b\binom{k+r-2}{r-1}^{-1}+s} (n-x-1)!(2k+2r-4 )^{x}}{(s/e)^sC^b (n-1)!}\\&\le \frac{2^bc^b\left((2k+2r-3)e\right)^{b\binom{k+r-2}{r-1}^{-1}}}{C^b}\cdot\left(\frac{(2k+2r-3)e^2}{s}\right)^s,
\end{aligned}
\]
which is at most $C':=((2k+2r-3)e^2)^{(2k+2r-3)e^2}$ if $C\ge 4er^{2}\binom{k+r-2}{r-1}((2k+2r-3)e)^{\binom{k+r-2}{r-1}^{-1}}$.
Recall that $|Q_{n}|=(n-1)!/2$, we get 
\[
\frac{\sum_{b=1}^{m}N_{\sigma}(b)p^{-b}}{|Q_n|} = \frac{\sum_{b=1}^{m} \sum_{s=1}^b 2N_{\sigma}(b,s)p^{-b}}{(n-1)!} \le \sum_{b=1}^{m} \frac{b C'}{2^{b}} = O_{r,k}(1).
\]

Moreover, if $C \to \infty$ when $n\to \infty$, then we can obtain that
\[
2^b \frac{2N_{\sigma}(b,s)p^{-b}}{(n-1)!}\le \frac{C'2^bc^b\left((2k+2r-3)e\right)^{b\binom{k+r-2}{r-1}^{-1}}}{C^b}=o(1).
\]
Let $2^b \frac{2N_{\sigma}(b,s)p^{-b}}{(n-1)!}=1/\omega$, where $\omega=\omega(n)\to \infty$, then we have
\[
\frac{\sum_{b=1}^{m}N_{\sigma}(b)p^{-b}}{|Q_n|} = \frac{\sum_{b=1}^{m} \sum_{s=1}^b 2N_{\sigma}(b,s)p^{-b}}{(n-1)!} = \sum_{b=1}^{m} \frac{b}{\omega2^{b}} = o(1).
\]
The proof is completed.
\end{proof}

\begin{remark}
We remark that the only place that we need $r\ge 3$ is in the proof of Lemma~\ref{k}. 
\end{remark}

\bibliographystyle{siam}
\bibliography{PowersHC}

\begin{bibdiv}
\begin{biblist}

\bib{BHKM2019}{article}{
      author={Bedenknecht, Wiebke},
      author={Han, Jie},
      author={Kohayakawa, Yoshiharu},
      author={Mota, Guilherme~O.},
       title={Powers of tight {H}amilton cycles in randomly perturbed
  hypergraphs},
        date={2019},
        ISSN={1042-9832},
     journal={Random Structures Algorithms},
      volume={55},
      number={4},
       pages={795\ndash 807},
         url={https://doi.org/10.1002/rsa.20885},
      review={\MR{4025389}},
}

\bib{Dirac}{article}{
      author={Dirac, G.~A.},
       title={Some theorems on abstract graphs},
        date={1952},
        ISSN={0024-6115},
     journal={Proc. London Math. Soc. (3)},
      volume={2},
       pages={69\ndash 81},
         url={https://doi.org/10.1112/plms/s3-2.1.69},
      review={\MR{47308}},
}

\bib{DF2013}{article}{
      author={Dudek, Andrzej},
      author={Frieze, Alan},
       title={Tight {H}amilton cycles in random uniform hypergraphs},
        date={2013},
        ISSN={1042-9832},
     journal={Random Structures Algorithms},
      volume={42},
      number={3},
       pages={374\ndash 385},
         url={https://doi.org/10.1002/rsa.20404},
      review={\MR{3039684}},
}

\bib{FSST22}{article}{
      author={Fischer, Manuela},
      author={\v{S}kori\'{c}, Nemanja},
      author={Steger, Angelika},
      author={Truji\'{c}, Milo\v{s}},
       title={Triangle resilience of the square of a {H}amilton cycle in random
  graphs},
        date={2022},
        ISSN={0095-8956},
     journal={J. Combin. Theory Ser. B},
      volume={152},
       pages={171\ndash 220},
         url={https://doi.org/10.1016/j.jctb.2021.09.005},
      review={\MR{4324910}},
}

\bib{Friedgut2005}{article}{
      author={Friedgut, Ehud},
       title={Hunting for sharp thresholds},
        date={2005},
        ISSN={1042-9832},
     journal={Random Structures Algorithms},
      volume={26},
      number={1-2},
       pages={37\ndash 51},
         url={https://doi.org/10.1002/rsa.20042},
      review={\MR{2116574}},
}

\bib{Frieze10}{article}{
      author={Frieze, Alan},
       title={Loose {H}amilton cycles in random 3-uniform hypergraphs},
        date={2010},
     journal={Electron. J. Combin.},
      volume={17},
      number={1},
       pages={Note 28, 4},
         url={https://doi.org/10.37236/477},
      review={\MR{2651737}},
}

\bib{Janson}{book}{
      author={Janson, Svante},
      author={{\L}uczak, Tomasz},
      author={Ruci{\'n}ski, Andrzej},
       title={Random graphs},
   publisher={Wiley-Interscience, New York},
        date={2000},
        ISBN={0-471-17541-2},
      review={\MR{2001k:05180}},
}

\bib{KNP2021}{article}{
      author={Kahn, Jeff},
      author={Narayanan, Bhargav},
      author={Park, Jinyoung},
       title={The threshold for the square of a {H}amilton cycle},
        date={2021},
        ISSN={0002-9939},
     journal={Proc. Amer. Math. Soc.},
      volume={149},
      number={8},
       pages={3201\ndash 3208},
         url={https://doi.org/10.1090/proc/15419},
      review={\MR{4273128}},
}

\bib{K1972}{inproceedings}{
      author={Karp, Richard~M.},
       title={Reducibility among combinatorial problems},
        date={1972},
   booktitle={Complexity of computer computations ({P}roc. {S}ympos., {IBM}
  {T}homas {J}. {W}atson {R}es. {C}enter, {Y}orktown {H}eights, {N}.{Y}.,
  1972)},
       pages={85\ndash 103},
      review={\MR{0378476}},
}

\bib{KS1983}{article}{
      author={Koml\'{o}s, J\'{a}nos},
      author={Szemer\'{e}di, Endre},
       title={Limit distribution for the existence of {H}amiltonian cycles in a
  random graph},
        date={1983},
        ISSN={0012-365X},
     journal={Discrete Math.},
      volume={43},
      number={1},
       pages={55\ndash 63},
         url={https://doi.org/10.1016/0012-365X(83)90021-3},
      review={\MR{680304}},
}

\bib{Kors}{article}{
      author={Kor\v{s}unov, A.~D.},
       title={Solution of a problem of {P}. {E}rd{\H{o}}s and {A}.
  {R}{\'{e}}nyi on {H}amiltonian cycles in nonoriented graphs},
        date={1977},
     journal={Diskretn. Anal. (Metody Diskret. Anal. v Teorii
  Upravljaju\v{s}\v{c}ih Sistem)},
      number={31},
       pages={17\ndash 56, 90},
      review={\MR{543833}},
}

\bib{KO2012}{article}{
      author={K\"{u}hn, Daniela},
      author={Osthus, Deryk},
       title={On {P}\'{o}sa's conjecture for random graphs},
        date={2012},
        ISSN={0895-4801},
     journal={SIAM J. Discrete Math.},
      volume={26},
      number={3},
       pages={1440\ndash 1457},
         url={https://doi.org/10.1137/120871729},
      review={\MR{3022146}},
}

\bib{NS2020}{article}{
      author={Narayanan, Bhargav},
      author={Schacht, Mathias},
       title={Sharp thresholds for nonlinear {H}amiltonian cycles in
  hypergraphs},
        date={2020},
        ISSN={1042-9832},
     journal={Random Structures Algorithms},
      volume={57},
      number={1},
       pages={244\ndash 255},
         url={https://doi.org/10.1002/rsa.20919},
      review={\MR{4120599}},
}

\bib{NS2019}{article}{
      author={Nenadov, Rajko},
      author={\v{S}kori\'{c}, Nemanja},
       title={Powers of {H}amilton cycles in random graphs and tight {H}amilton
  cycles in random hypergraphs},
        date={2019},
        ISSN={1042-9832},
     journal={Random Structures Algorithms},
      volume={54},
      number={1},
       pages={187\ndash 208},
         url={https://doi.org/10.1002/rsa.20782},
      review={\MR{3884618}},
}

\bib{PP2016}{article}{
      author={Parczyk, Olaf},
      author={Person, Yury},
       title={Spanning structures and universality in sparse hypergraphs},
        date={2016},
        ISSN={1042-9832},
     journal={Random Structures Algorithms},
      volume={49},
      number={4},
       pages={819\ndash 844},
         url={https://doi.org/10.1002/rsa.20690},
      review={\MR{3570989}},
}

\bib{Posa1}{article}{
      author={P\'{o}sa, L.},
       title={Hamiltonian circuits in random graphs},
        date={1976},
        ISSN={0012-365X},
     journal={Discrete Math.},
      volume={14},
      number={4},
       pages={359\ndash 364},
         url={https://doi.org/10.1016/0012-365X(76)90068-6},
      review={\MR{389666}},
}

\bib{R2000}{article}{
      author={Riordan, Oliver},
       title={Spanning subgraphs of random graphs},
        date={2000},
        ISSN={0963-5483},
     journal={Combin. Probab. Comput.},
      volume={9},
      number={2},
       pages={125\ndash 148},
         url={https://doi.org/10.1017/S0963548399004150},
      review={\MR{1762785}},
}

\bib{RR2010}{incollection}{
      author={R\"{o}dl, Vojtech},
      author={Ruci\'{n}ski, Andrzej},
       title={Dirac-type questions for hypergraphs---a survey (or more problems
  for {E}ndre to solve)},
        date={2010},
   booktitle={An irregular mind},
      series={Bolyai Soc. Math. Stud.},
      volume={21},
   publisher={J\'{a}nos Bolyai Math. Soc., Budapest},
       pages={561\ndash 590},
         url={https://doi.org/10.1007/978-3-642-14444-8_16},
      review={\MR{2815614}},
}

\bib{Zhao2016}{incollection}{
      author={Zhao, Yi},
       title={Recent advances on {D}irac-type problems for hypergraphs},
        date={2016},
   booktitle={Recent trends in combinatorics},
      series={IMA Vol. Math. Appl.},
      volume={159},
   publisher={Springer, [Cham]},
       pages={145\ndash 165},
         url={https://doi.org/10.1007/978-3-319-24298-9_6},
      review={\MR{3526407}},
}

\end{biblist}
\end{bibdiv}

\end{document}